\documentclass[coll]{imsart}
\usepackage[longnamesfirst]{natbib}
\usepackage{geometry}
\usepackage{color}
\usepackage{tabularx}
\usepackage{algorithm}
\usepackage{hyperref}
\usepackage{graphicx}
\arxiv{arXiv:1206.6913v1}
\usepackage{amssymb,amsmath,amsthm} 
\theoremstyle{plain} \newtheorem{thm}{Theorem}
\newtheorem{lemma}{Lemma}
\newtheorem{prop}{Proposition}
\theoremstyle{remark} 
\theoremstyle{definition} 

\definecolor{bgblue}{rgb}{0.04,0.39,0.53}
\usepackage[longnamesfirst]{natbib}

\renewcommand\labelenumi{(\theenumi)}
\newcommand{\CM}{\mathcal M}
\newcommand{\Mp}{{\mathcal M}_{\mathbf p}}
\newcommand{\CN}{\mathcal N}
\newcommand{\CU}{\mathcal U}
\newcommand{\BBR}{\mathbb{R}}
\newcommand{\BH}{\bar{\cal H}}
\newcommand{\bx}{\boldsymbol{x}}

\def\noi{\noindent}
\numberwithin{equation}{section}
\begin{document}
\begin{frontmatter}
\title{Sampling From A Manifold
\protect\thanksref{T1}}
\runtitle{Sampling From A Manifold}
\thankstext{T1}{This work was part of a project funded by the French ANR under a Chaire d'Excellence
at the University of Nice Sophia-Antipolis.}
\date{August 6, 2011}
\begin{aug}
\author{\fnms{Persi} \snm{Diaconis}\thanksref{t3}\ead[label=e1]{No email}},
\author{\fnms{Susan} \snm{Holmes}\thanksref{t1}
\ead[label=e3]{susan@stat.stanford.edu}
}
\and
\author{\fnms{Mehrdad} \snm{Shahshahani}\thanksref{t2}\ead[label=e2]{mshahshahani@gmail.com}
\ead[label=u2,url]{http:}}
\thankstext{t2}{Supported by a DARPA grant HR 0011-04-1-0025.}
\thankstext{t1}{Supported by the NIH grant NIH-R01GM086884.}
\thankstext{t3}{Supported by NSF grant DMS  0804324}
\runauthor{Diaconis, Holmes, Shahshahani}
\affiliation{Stanford University and U. Teheran}
\address{Susan Holmes and Persi Diaconis\\ Department of Statistics\\
Sequoia Hall\\
CA 94305 Stanford, USA.\\
\printead{e3}}
\address{Mehrdad Shahshahani\\
Mathematics Institute\\
Teheran, Iran.\\
\printead{e2}\\
}
\end{aug}
\begin{abstract}
  We develop algorithms for sampling from a probability distribution
  on a submanifold embedded in $\mathbb{R}^n$. Applications are given
  to the evaluation of algorithms in `Topological Statistics';  to
  goodness of fit tests in exponential families and to Neyman's smooth test.
  This article is partially expository, giving an introduction to the tools of geometric
  measure theory. 
  
\end{abstract}

\begin{keyword}[class=AMS]
\kwd[Primary ]{60K35}
\kwd{60K35}
\kwd[; secondary ]{60K35}
\end{keyword}

\begin{keyword}
\kwd{manifold}
\kwd{conditional distribution}
\kwd{geometric measure theory}
\kwd{sampling}
\end{keyword}


\end{frontmatter}

\section{Introduction}\label{INT}
A variety of inferential tasks require drawing  samples from a probability 
distribution on a manifold. This occurs in sampling from the posterior 
distribution on constrained parameter spaces (eg covariance 
matrices), in testing goodness of fit for exponential families
conditional on sufficient statistics (eg the sum and product of the observations in a Gamma
family), and in generating data to test algorithms
in tolopogical statistics.

In our applications, we found that
examples involved domains with corners
and non smooth functions (eg max$(|x_1|,|x_2|,\ldots, |x_n|)$). We found
a useful set of tools in geometric measure theory. One of our goals is to explain and illustrate 
this in the usual language of probability and statistics.

To introduce the subject, consider the following two examples, used as
illustrations throughout.

\medskip
\noi \textbf{Example 1A: The Curved Torus} \quad Figure 2 shows a picture
of 1000 points on the torus
\begin{equation}
\CM = \{ [ ( R+r \cos ( \theta ) ) \cos ( \psi ) , ( R+r \cos ( \theta ) ) \sin ( \psi ) , r \sin ( \theta ) ] \},
\label{eq:1}
\end{equation}
$0 \leq \theta,\psi <2\pi$ for $R > r> 0$. 
The torus is formed by taking a circle of radius $r$ in the $(x,z)$ plane, centered at $x=r,z=0$ and rotating it around the $z$ axis.

 
Formula (\ref{eq:1}) gives the embedding of $\CM$ as a compact
2-dimensional manifold in $\BBR^3$. As such, $\CM$ inherits a natural
area measure: roughly, take a region on $\CM$, thicken it out by
$\epsilon$ to be fully 3-dimensional, compute the usual volume of the
thickened region and take the limit of this area divided by $\epsilon$
as $\epsilon \longrightarrow 0$.  This area measure can be normalized
to be a probability measure $\BH^2(dx)$ on $\CM$. The points shown are
sampled from $\BH^2(dx)$.
\smallskip
\begin{center}
\includegraphics[width=3in]{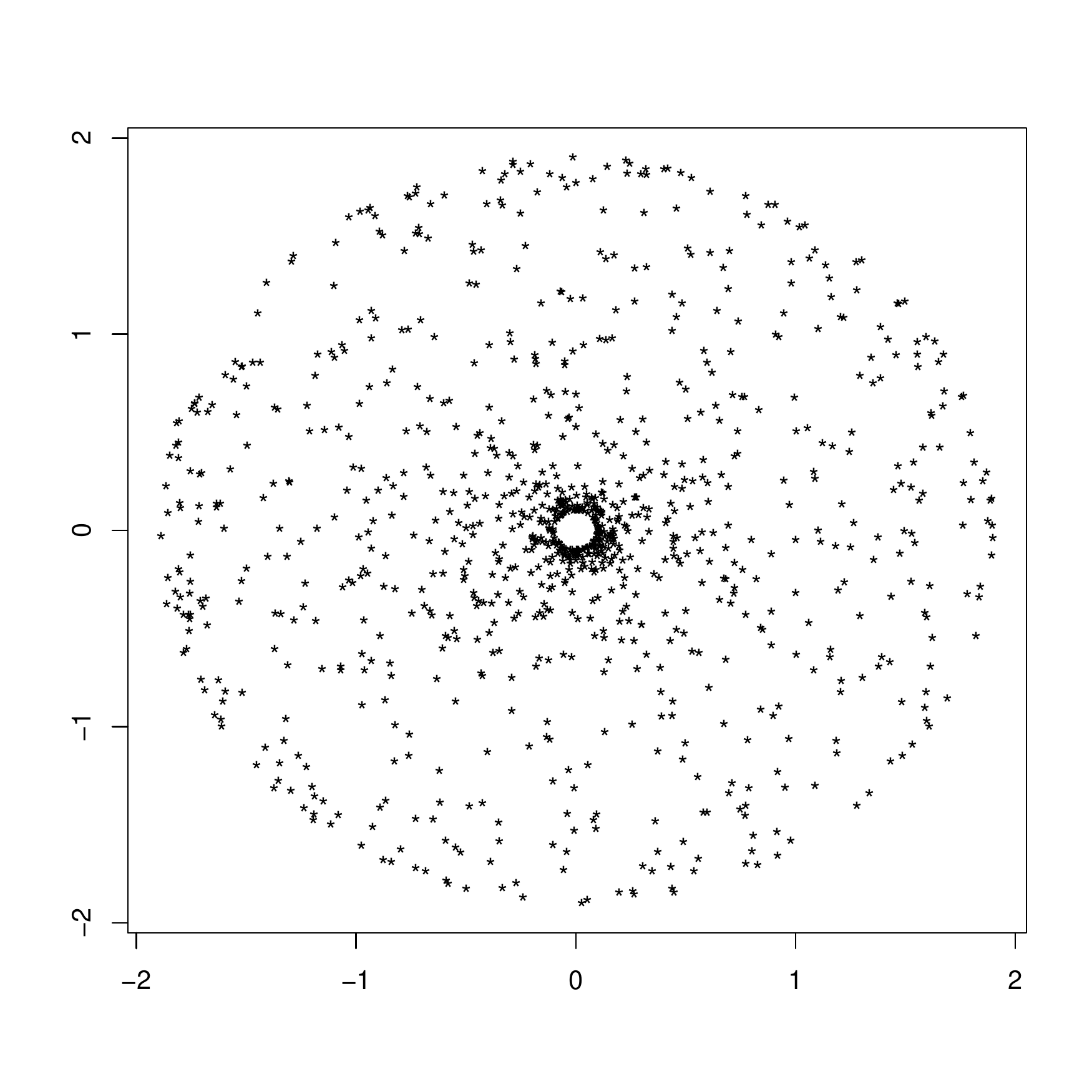}\\
{Figure 1: A sample of 1000 points from the na\"ive measure on a torus with R=1,r=0.9}
 \end{center}

Note that the sampled points are denser in regions with higher curvature
such as the inside of the torus. This distribution is from the na\"ive choice:
choose $\theta$ and $\psi$ uniformly 
and map onto $\CM$ using (\ref{eq:1}).
Figure (2.3) show both correctly and incorrectly generated points, see next section.

Such samples, with noise added, are used to calibrate topological
algorithms for estimating dimension, number of components and homology
in the emerging field of topological statistics. 
Examples such as two linked tori on the seven sphere and Klein bottles
are shown to come up naturally in image analysis \citep{CarlssonEGV}.

\medskip
\noi \textbf{Example 1B: Testing the Gamma Distribution} \quad For fixed
$n \geq 3, S, P>0$, let
\begin{equation}
\CM = \Big \{ (x_1,\ldots,x_n); \ \ x_i >0, \ \sum_{i=1}^n x_i=S, \ \prod_{i=1}^n x_i=P \Big \}.
\end{equation}
This is a compact ($n-2$)-dimensional submanifold in $\BBR^n$. The need
for samples from $\CM$ comes up in testing if random variables $X_1,
X_2, \ldots, X_n$ are independently drawn from the Gamma density
\begin{equation}
\frac{e^{-x/\sigma} x^{a-1}}{\sigma^a\Gamma(a)} \qquad 0<x<\infty,
\end{equation}
with $\sigma, a >0$ unknown parameters. The sufficient statistics for
$\sigma, a$ are $S=\sum_{i=1}^n X_i, P=\prod_{i=1}^n X_i$. In
numerous writings, R.~A. Fisher suggested using the conditional
distribution of $X_1, \ldots, X_n$ given $S, P$ to give exact goodness
of fit tests.
These ideas are reviewed in section 3 below. The conditional
distribution has a simple density with respect to $\BH^{n-2}(dx)$
 leading to practical algorithms for random generation and testing.
The proposed tests are different than the ones in \cite{Kallioras} or \cite{Pettitt}.
\cite{Goldman} and 
\cite{Yang:2006} explain interesting applications of these tests in modern
evolutionary analyses of DNA.
\subsubsection*{Related Literature}
There has been a steady interest in statistics on manifolds. 
The development of mean and variance estimators appears in \cite{Pennec}
and \cite{bhattacharya2003large}.
The book by \cite{bhattacharyab}  
about data on the shape space manifold contains several interesting
results.
Data on the sphere and the projective space are discussed in \cite{Beran}, \cite{Fisher} and \cite{Watson}.
Data on more general manifolds  appear in \cite{Gine}. One widespread
example occurs in physics and chemistry problems involving
configurations of atoms with some inter-atomic distances or angles
fixed; see \cite{Fixman:1974} or \cite{Ciccotti}. Any of these
settings give rise to the need for Monte Carlo sampling on manifolds.

There are well-known algorithms for sampling from the uniform
distribution on compact groups and other homogeneous spaces.
For instance, \cite{Eaton} proves that if an $n\times n$ matrix is
filled with iid standard normals and the QR decomposition is carried out, then 
the $Q$ part is distributed as the uniform distribution on the orthogonal group (Haar measure).
 \cite{Francesco,DiaconisS} develop this. 
There are also elegant algorithms
for sampling from the boundary of compact, convex sets in $\BBR^n$
\citep{hit,shake}. 
A different procedure, the \cite{Lalley} ``princess-and monster''
algorithm has been studied for sampling from the boundaries
of more general sets \citep{Comets,Niyogi}. 
These algorithms are based on moving within the interior
of the bounded set reflecting off the boundary. They are
different from the present procedures and may be very effective when applicable.
We do not know previous
literature on sampling from more general manifolds.

 Of course, conditional probability densities are standard fare, even
 with very general conditioning. However, explicit description of area
 measure and the use of the co-area formula is not so common. We only
 know of the foundational monograph by \cite{Tjur:1974}. This contains
 a good history. The development is both more and less general. Tjur
 works with Riemannian manifolds and smooth functions. We work with
 embedded manifolds but allow Lipschitz functions such as max/min. Tjur
 gives a self-contained development based on Radon measures. We are
 able to use more classical foundations from standard sources. Tjur's
 valuable monograph was written before the computer revolution. We
 emphasize techniques useful for sampling.

This paper studies the following problem of sampling from 
  $\CM$, an $m$-dimensional submanifold in $\BBR^n$.  Consider $f(x)\geq 0$
such that $\int_{\CM}f(x)\BH^m(dx) <\infty$ with $\BH^m(dx)$ the $m$-dimensional
area measure on $\CM$. Samples are to be drawn from the normalized
version of $f$. Section 2 gives basic definitions for submanifolds,
area measure,  Jacobians and the co-area formula. 
These notions are illustrated on examples 1A,1B.

Section 3 develops the theory for exponential families,
Section 4 that of Neyman's smooth test.

The algorithms presented are reasonably standard Markov chain Monte
Carlo methods supplemented by some geometrical tricks and the tools of
geometric measure theory. We hope they will be useful to researchers
who face similar problems.

The subject developed here may be considered as a continuous
analog of algebraic statistics as initiated in \cite{DiaconisSturmfels} and reviewed in
\cite{Drton}.
That theory began  by developing algorithms
for sampling from the conditional distribution
of discrete exponential families given their sufficient statistics. There, finding ways
of moving around on the space of data sets with given sufficient statistics 
leaned on tools from computational algebra (Gr\"obner bases).
Here, the same task is studied using direct geometric analysis 
and tools such as the curve selection lemma.
\section{Definitions and Tools}\label{DEF}
 
The classical subject of calculus on manifolds has an enormous
expository literature. We have found the elementary treatment of
 \cite{Hubbard:2007} readable and useful. In our
applications, pieces of manifolds with corners occur naturally. For
example, testing the three-parameter Gamma density gives rise to
\begin{equation*}
\CM = \Big \{ (x_1, \ldots, x_n); \ \ x_i >0, \ \sum_{i=1}^n x_i= S,
\ \prod_{i=1}^n x_i= P, \ \min x_i \geq m \Big \}.
\end{equation*}
Geometric measure theory provides convenient tools. We use \cite{Federer:1996},
denoted [F], as a principle reference. The introductory account by
\cite{Morgan} gives a useful taste of the subject matter. 
Recent references are \cite{Mattila},
\cite{Krantz}.

\subsection{First Definitions}\label{DEF.a}

A function $f:\BBR^{m}\longrightarrow \BBR^{n}$ is {\em
  Lipschitz} if $|f(x) - f(y)| \leq c |x-y|$ for some finite, positive
$c$. Euclidean distance is used for $|\cdot|$ on both sides. A set
in $\BBR^n$ is $m$-rectifiable [F, p. 251] if it is the
Lipschitz image of a bounded subset in $\BBR^m$.
This is a very rich class of sets, discussed at
length in the references above. All of the sets that arise in our
applications are rectifiable.

Use $\lambda^n(dx)$ for Lebesgue measure on the Lebesgue measurable
sets of $\BBR^n$. Given any subset $A \subseteq \BBR^n$, define the  $m$-dimensional
{\em Hausdorff measure} $\BH^m(A)$ by
\begin{equation*}
\BH^m(A) = \lim_{\delta \to 0} \inf_{\substack{A \subseteq \cup S_i, \\  \text{diam} (S_i) \leq \delta}} \
  \sum \alpha_m \Big ( \frac{\text{diam}(S_i)}{2} \Big ) ^m
\end{equation*}

\noi The infimum is taken over all countable coverings {$S_i$} of $A$ with
$\text{diam}(S_i) = \sup \{ |x-y| : x, y \in S_i \}$ and $\alpha_m =
\Gamma (\frac{1}{2})^m / \Gamma [(\frac{m}{2}) +1]$, the volume of the
unit ball in $\BBR^m$. Hausdorff measure is an outer measure which is
countably additive on the Borel sets of $\BBR^n$. It serves as area
measure for subsets. If the set $A$ above is $m$-rectifiable, the
coverings above can be restricted to balls or cubes [F,
Sect.~3.2.26]. For a closed set $A$, [F, Sect.~3.2.39] shows $\BH^m(A)
= \lim_{\epsilon \to 0} \lambda^n \{ x: \text{dist}(x,A) < \epsilon \}/
\alpha_{(n-m)} \epsilon^{n-m} $, thus justifying the heuristic
definition of area measure in Example A of Section 1.

To actually compute area measure, the Jacobian is an essential tool.
Call $f: \BBR^m \to \BBR^n$ differentiable at $x \in \BBR^m$ if there
exists a linear map $L: \BBR^m \to \BBR^n$ with
\begin{equation*}
\lim_{h \to 0} |f(x+h) - f(x) - L(h)| / |h| = 0.
\end{equation*}
The linear map $L$ is denoted $Df(x)$ when it exists. A celebrated
theorem of Rademacher [F, Sect.~3.1.6] says that a Lipschitz function
is differentiable at $\lambda^m$ a.e. $x \in \BBR^m$. For a
differentiable function, $Df$ can be computed using partial
derivatives $D_i(x) = \lim_{h \to 0} (f(x_1, \ldots, x_i+h, \ldots,
x_m) - f(x))/h$. As usual, the derivative matrix is
\begin{equation*}
(D f(x) )_{i, j} = D_i f_j(x) \qquad 1 \leq i \leq m, \ 1 \leq j \leq n
\end{equation*}

\noi If $f: \BBR^m \to \BBR^n$ is differentiable at $x$, the $k$-dimensional
Jacobian $J_k f(x)$ may be defined as the 
norm of the derivative matrix [F, page 241]. Geometrically 
$J_k f(x)$ is defined as the
maximum $k$-dimensional volume
of the image under $Df(x)$ of a unit $k$-dimensional cube in $\BBR^m$
(the maximum over all possible rotations of the cube under orthogonal
rotations in ${\cal O}_m$ \cite[p. 25]{Morgan}). As usual, if rank $Df(x) <
k,\ J_kf(x) = 0$. If rank $Df(x) = k$, then $(J_kf(x))^2$ equals the sum
of squares of the determinants of the $k \times k$ submatrices of
$Df(x)$. Usually, $k=m$ or $n$. Then $(J_kf(x))^2$ equals the
determinant of the $k \times k$ product of $Df(x)$ and its
transpose. If $k=m=n, \ J_kf(x)$ is the absolute value of the
determinant of $Df(x)$.

\subsection{The Area Formula}\label{DEF.b}

The basic area formula [F, Sect.~3.2.5] is a useful extension of the
  classical change of variables formula of calculus.

\noi \textbf{Theorem: Area Formula} \qquad {\em If $f: \BBR^m \to \BBR^n$ is
Lipschitz and $m \leq n$, then
\begin{equation}
\int_A g(f(x))J_m f(x) \lambda^m (dx) = \int_{\BBR^n} g(y) N (f | A, y)
\BH^m (dy)
\end{equation}
whenever $A$ is $\lambda^m$ measurable, $g: \BBR^n \to \BBR$ is Borel,
and $N (f | A, y) = \# \{ x \in A : f(x) = y \}$.}
\medskip

\noi\textit{Remarks}

1. In this paper, $f$ is usually a parameterization of a submanifold
$\CM$, so $f$ is $1-1$ and the right-hand integral is the surface area
integral of $g$ over $f(A)$. The left side shows how to carry out this
integral using Lebesgue measure on $\BBR^m$ and the Jacobian. It shows
that sampling from the density $J_mf(x)$ (normalized) on $\BBR^m$ and
then mapping onto $\CM$ via $f$ gives a sample from the area measure.
\smallskip

2. There are many extensions and refinements of the area formula [F,
  Sect.~3.2]. In particular [F, Sect.~3.2.20] extends things to
  approximately differentiable functions and [F, Sect.~3.2.46] extends
  from Euclidean space to Riemannian manifolds.  \medskip

\noi \textbf{Example 1A continued: The Curved Torus} \quad For the
parameterization given in Example 1A, the curved torus is the Lipschitz
image of $\{ 0\leq\theta, \psi<2\pi \}$, with
\begin{equation}
f(\theta,\psi)=  (R+ r \cos (\theta)) \cos (\psi), ( R+ r
\cos(\theta)) \sin (\psi), \ r \sin (\theta))
\end{equation}
\begin{equation}
Df(\theta,\psi) = \left[ \begin{array}{cc}
 -  r  \sin (\theta) \cos (\psi) &- ( R+ r \cos(\theta)) \sin (\psi) \\
\noalign{\medskip}- r \sin (\theta) \sin (\psi) & ( R+ r \cos (\theta) ) \cos (\psi) \\
\noalign{\medskip} r \cos (\theta) &0
\end{array} \right] 
\end{equation}
\begin{equation}
J_2^2f(\theta,\psi) = \text{det} \left[ \begin {array}{cc}
r^2 & 0\\ \noalign{\medskip} 0 & (R+r \cos(\theta))^2
\end {array} \right] = r^2(R+r \cos(\theta))^2
\end{equation}
\medskip

As explained in
Section~\ref{DEF}, $\CM$ is parametrized by ${\cal U}= \{ \theta,
\psi, \ 0 \leq \theta, \ \psi < 2\pi \}$ and the task reduces to
sampling $(\theta, \psi)$ from the density $g(\theta, \psi) = (\frac{1}{
4\pi^2} )(1+ (r/R) \cos \theta)$. A random number generator outputs
points that we assume are uniformly distributed on $[0,1]$ and
the task reduces to converting these into a sample from $g$. From the
form of $g$, the measure factors into the uniform density for $\psi$
on $[0,2\pi)$ and the density
\begin{equation*}
g_1(\theta)= \frac{1}{2\pi} \Big ( 1+ \frac{r}{R} \cos \theta \Big ) \quad
\text{on } 0 \leq \theta < 2\pi.
\end{equation*}

\noi We may sample points from $g_1$ by rejection sampling
\citep{Hammersley:1964}. The function $(1 + (r/R) \cos \theta)$ is enclosed
in the box $0 \leq \theta < 2\pi, [1 - (r/R) < \eta < 1+ (r/R)]$. Choose
points $(\theta, \eta)$ uniformly in this box from two-dimensional
Lebesgue measure. This uses two calls to the underlying uniform random
number generator. If $\eta < 1 + (r/R) \cos \theta$, output $\theta$. If
not, choose again, continuing until the condition holds. The resulting
$\theta$ is distributed as $g_1$. Sample code for this is in algorithm 1.

\begin{algorithm}
\begin{verbatim}
            reject=function(n=100,r=0.5,R=1){
#Rejection sampler
            xvec=runif(n,0,2*pi)
            yvec=runif(n,0,1/pi)
            fx=(1+(r/R)*cos(xvec))/(2*pi)
             return(xvec[yvec<fx])   }
\end{verbatim}
\caption{Rejection Sampling yielding $g_1$.}
\label{alg:display1}
\end{algorithm}

What we get is a density with support  $[0,2\pi]$.
See Figures 1 and 2 below.
\smallskip
\begin{figure}[h!]
\begin{center}
   \includegraphics[width=3in]{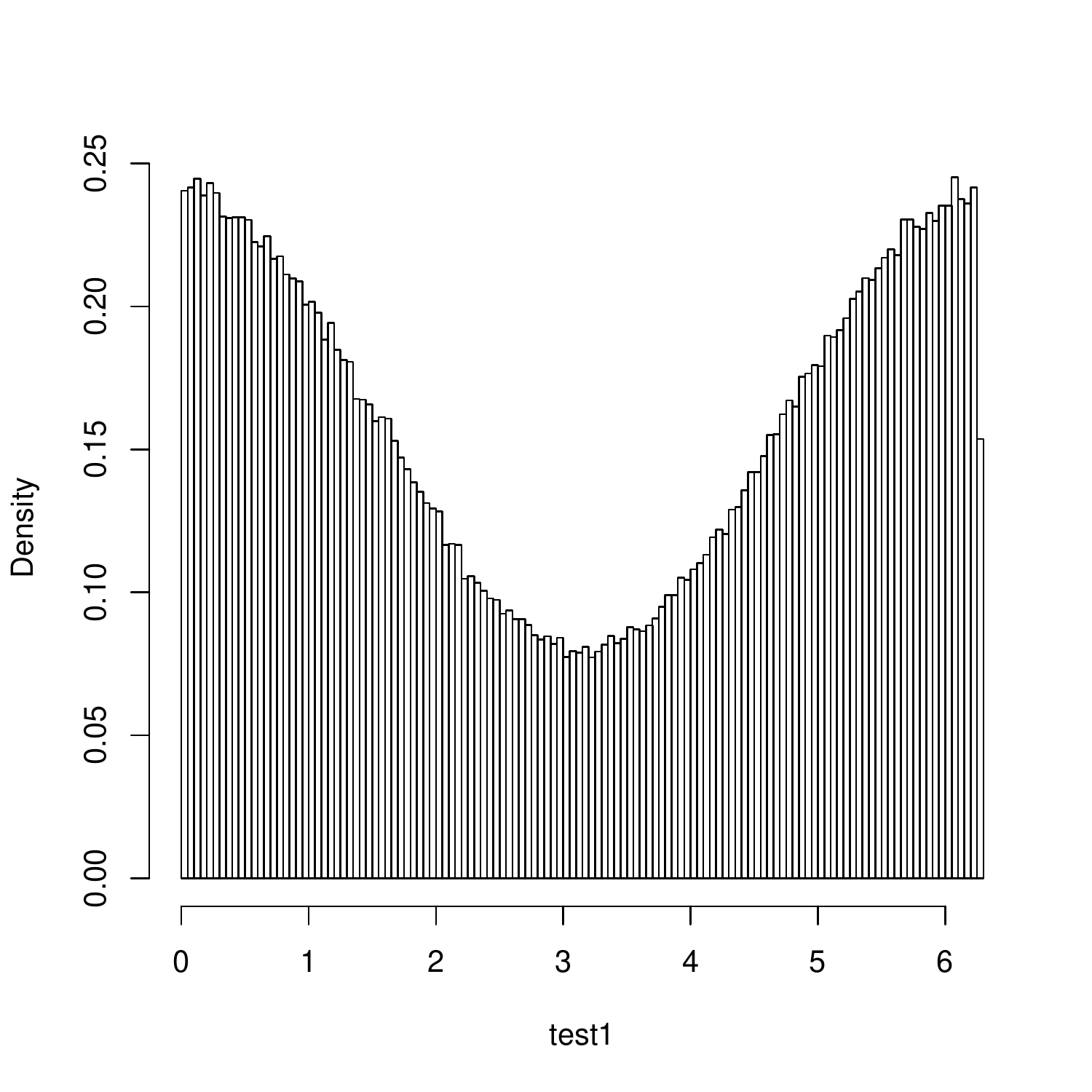}\\
   
 \end{center}
 \caption{Rejection sampling density proportional to $1+\frac{r}{R}cos(\theta)$} 
 \end{figure}

%
\smallskip

\smallskip
\begin{figure}
\begin{center}
     \includegraphics[width=4in]{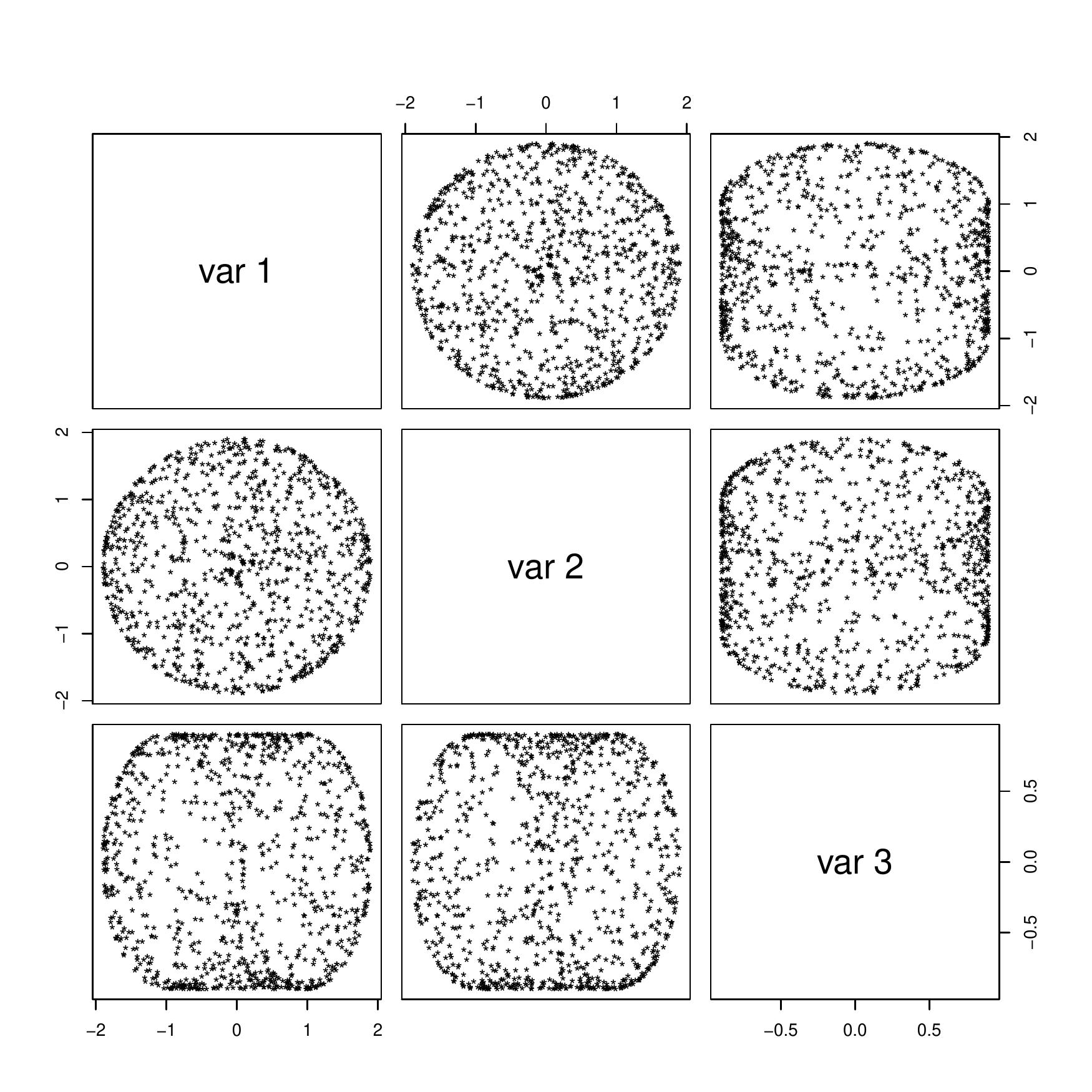}\\
   {Correctly generated points uniformly on the torus}
 \end{center}
\begin{center}
     \includegraphics[width=4in]{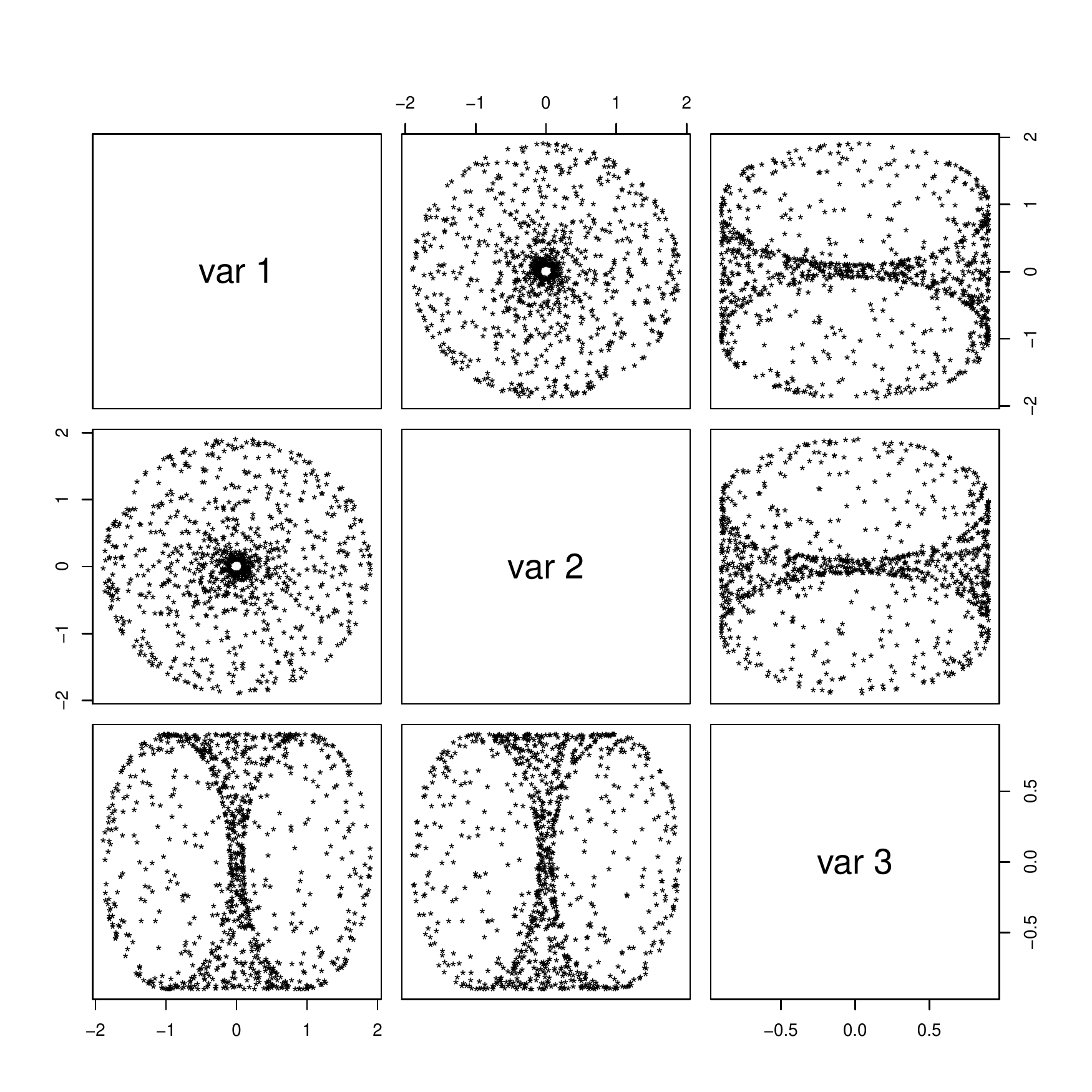}\\
   {Basic Uniform on Parameters. }
 \end{center}
 \caption{Top figure shows a 3D representation of a sample of size 1000 with parameters
 $R=1$, $r=0.9$, the lower figure shows the incorrectly sampled points, although the difference
 is not obvious visually, standard tests pick up the difference between the two distributions.}
 \end{figure}

\noi \textbf{Example 1B continued: Sum and Product Fixed} \quad Here
\begin{equation*}
\CM = \Big \{ (x_1,\ldots,x_n); \ \ x_i >0, \ \sum_{i=1}^n x_i=S, \ \prod_{i=1}^n x_i=P \Big \}.
\end{equation*}

\noi The constraints $S,P$ satisfy $0<P^{1/n} \leq S/n$ because of the
arithmetic-geometric mean inequality. Any such $S,P$ can occur. To
find a parameterization of $\CM$ consider the projection
\begin{equation*}
\begin{aligned}
\Pi : \CM &\to \BBR^{n-2}\\
(x_1, \ldots, x_n) &\to (x_3, \ldots, x_n)
\end{aligned}
\end{equation*}

\noi Let $s = x_3 + \cdots + x_n = S-t$ with $t \geq 0$ and $x_3 x_4
\cdots x_n = p$. The equations $x_1 + x_2 =t, x_1 x_2 = P/p$ have a
positive real solution if and only if $t^2 \geq 4P/p$. In this case
the solution is the pair
\begin{equation*}
\{ x_1, x_2 \} = \bigg ( t \pm \sqrt{t^2 - \frac{4P}{p}} \bigg ) \Big / 2.
\end{equation*}

\noi One way to parametrize $\CM$ is to define
\begin{equation}
\CM^+ = \{ x \in \CM : x_1 \geq x_2 \}, \quad \CM^- = \{ x \in \CM :
x_1 < x_2 \}
\label{eq2.5}
\end{equation}

\noi Define
${\cal U} = \{ (x_3, \ldots, x_n) : x_i >0, s < S, p < 4P/(S-4)^2 \}
\quad s = \sum_{i=3}^n x_i, \ p= \prod_{i=3}^n x_i;$

 $f: {\cal U} \to \CM^+$ is defined by
\begin{equation}
f(x_3, \ldots, x_n) = (f_1(x_3, \ldots, x_n ),f_2(x_3, \ldots, x_n ),x_3, \ldots, x_n )
\label{2.6}
\end{equation}
$$\mbox{with } f_1(x_3, \ldots, x_n) =  \frac{(S-\sum_{i=3}^nx_i) + \sqrt{(S-\sum_{i=3}^nx_i)^2 -
    \frac{4P}{\prod_{i=3}^nx_i}}}{2}$$
$$\mbox{    and }
    f_2(x_3, \ldots, x_n) =  \frac{(S-\sum_{i=3}^nx_i) - \sqrt{(S-\sum_{i=3}^nx_i)^2 -
    \frac{4P}{\prod_{i=3}^nx_i}}}{2}
    $$

The derivative is the $n \times (n-2)$ matrix
\begin{equation}
Df = \left[
\begin{array}{cccc}
D_3f_1 & D_4f_1 & \cdots & D_nf_1 \\
D_3f_2 & D_4f_2 & \cdots & D_nf_2 \\
1 & 0 & \cdots & 0 \\
0 & 1 & \cdots & 0 \\
  &   & \vdots &   \\
0 & 0 & \cdots & 1 \\
\end{array} \right]
\label{eq2.1}
\end{equation}
$(J_{n-2}f(x))^2 = \text{det} ((Df)^TDf)$ is the determinant of a matrix
of form $I_{n-2} + V V^T + W W^T$ with $V^T, W^T$ the first and second
rows of $Df$. A well-known determinant identity reduces this to a $2
\times 2$ determinant; if $B$ is $p \times m$ and $C$ is $m \times p$
then
 $\text{det} (I_p + BC) = \text{det} (I_m + CB)$. It
follows that

 \begin{equation}
 ((J_{n-2}f(x))^2 = \text{det} \bigg ( I_2 + 
\begin{pmatrix}
V^TV & V^TW \\
V^TW & W^TW
\end{pmatrix} \bigg ).
\label{eq2.2}
\end{equation}
\smallskip

\noi To summarize:

\begin{prop}\quad The density of the $(n-2)$-dimensional
area measure $\BH^{n-2}$ on the submanifold $\CM^+$ in (\ref{eq2.5})
parametrized by $f: {\cal U} \to \CM^+$ is $J_{n-2}f(x)$ of
(\ref{eq2.2}), above
with $V,W$ the first two rows of matrix (\ref{eq2.1}).
\label{prop2.1}
\end{prop}
\medskip

\noi{\bf Remarks}

1. Up to sets of $\BH^{n-2}$ measure $0$, a similar result holds for
$\CM^-$. Since $\CM^-$ and $\CM^+$ patently have the same area
measure, it is easy to sample from $\CM$ using an additional coin flip
to randomize the first two coordinates.

2. Of course, any $(n-2)$-tuple of  coordinates can be used for the
parameterization. In practical simulation, it might be wise to sample
from $\CM$ as above and follow this by a random permutation of the
coordinates.

3. The function $f$ defined in (\ref{2.6}) , is only locally Lipschitz
because of $p$ in the denominator. However, ${\cal U}$ may be
decomposed into a countable union of pieces with $f$ Lipschitz on each
piece. Because the formula for $J_nf$ is local, the proposition is
true as stated. 

Given $S, P$, the manifold $\CM^+$ is parametrized
by $\cal U$ of Example 1B. The task of sampling from area measure on
$\CM$ is reduced to sampling from $J_{n-2}f(x)$ in $\cal U$. One
problem here is that although  $z=\int_{\cal U} J_{n-2}f(x) \lambda^{n-2}(dx) <
\infty$ and $J_{n-2}f/z$ is a probability density on $\cal U$, the
value of $z$ is unknown. This standard problem may be approached by
the Metropolis algorithm, the Gibbs sampler,
importance sampling, or by the hit-and-run algorithm in many
variations (see \cite{LiuBook}, \cite{HitandRun} for background). 
Here, we briefly explain the Metropolis
algorithm for sampling from $J_{n-2}f$. This generates a Markov chain $X_0,
X_1, X_2, \ldots$ starting from $X_0=x_0$, a fixed point in $\cal
U$. From $X_n=x \in {\cal U}$, we propose $y \in \BBR^M$, choosing
$y=x+\epsilon$ with $\epsilon$ chosen (say) uniformly from a unit cube
centered at $x$. Then,
\begin{equation*}
X_{n+1} = \begin{cases}
y & \text{ with probability } \min \Big(\frac{J_{n-2}f(y)}{J_{n-2}f(x)}, 1 \Big)\\
x & \text{ otherwise.} \end{cases}
\end{equation*}

\noi Since $J_{n-2}f(y)$ is taken as $0$ outside $\cal U$, note that $X_{n+1}
\in {\cal U}$. Standard theory shows that for $n$ large, $P(X_n \in A)
\sim \int_A \frac{J_{n-2}f(x)}{z} \lambda^M(dx)$. Careful evaluation of how
large $n$ must be to make this approximation valid is an open research
problem both here and in most real applications of the Metropolis
algorithm (see \cite{Metropolis} and \cite{DLM}). A host of heuristics are available for
monitoring convergence; for adapting the choice of the
proposal for $\epsilon$  and for efficient use of the output. We will not discuss these further here.

Several further examples admitting an explicit parameterization, with
computations of $Jf$, are in \citet[Chap. 5]{Hubbard:2007} which is
enthusiastically recommended to newcomers. 

\subsection{Conditional Densities and the Co-Area Formula}\label{DEF.c}

Federer's co-area formula gives an explicit density for the
conditional distribution. The main tool is:

\noi\textbf{Theorem: Co-Area Formula} [F, Sect.~3.2.12] \qquad {\em
  Suppose that $\Phi: \BBR^M \to \BBR^N$ is Lipschitz with $M >N$. Then
\begin{equation}
\int_{\BBR^M} g(x) J_N \Phi(x) \lambda^M (dx) = \int_{\BBR^N}
\int_{\Phi^{-1}(y)} g(x) \BH^{M-N}(dx) \lambda^N (dy).
\label{th1}
\end{equation}
}

\noi In (\ref{th1}), $g$ is Lebesgue measurable from $\BBR^M \to \BBR$ and
$J_N\Phi$ is defined in Section 2.1.

Recall next the definition of a regular conditional probability. Let
$(\Omega, {\cal F}, P)$ be a probability space and ${\cal C} \subseteq {\cal F}$
a sub-sigma algebra. A function $P(w,dw)$ from $(\Omega \times {\cal F})$
into $[0,1]$ is a regular conditional probability for $P$ given ${\cal
C}$ if
\begin{subequations}
\begin{align}
&\text{For each } w \in \Omega, P(w, \cdot) \text{ is a probability measure
  on } {\cal F}.\label{list.a}\\
&\text{For each } F \in {\cal F}, \text{ the function } w \mapsto P(w,
F) \text{ is } \cal C \text{ measurable.}\label{list.b}\\
&\text{For }  C \in {\cal C}, F \in {\cal F}, P(C \cap F) = \int_C
  P(w,F) P(dw).\label{list.c}
\end{align}
\end{subequations}

Let $p(x)$ be a probability density on $\BBR^M$ with respect to
$\lambda^M(dx)$. Let $\Phi: \BBR^M \to \BBR^N$ be Lipschitz with
$M>N$. From Rademacher's Theorem, $\Phi$ is differentiable at almost
every $x$, and $J_N \Phi(x)$ can be computed by the usual rules.

\begin{prop}
Suppose that $J_N \Phi(x)$ exists and is strictly positive for all $x$
where $p(x) > 0$. Then

\begin{enumerate}
\item The marginal density of $\Phi$ is absolutely continuous with
  density
\begin{equation*}
m(y) = \int_{\Phi^{-1}(y)} \frac{p(x)}{J_N\Phi(x)} \BH^{M-N} (dx)
\text{ with respect to } \lambda^M(dy).
\end{equation*}
\item If $m(y) \in \{ 0,\infty \}$, set ${\cal Q}(y,F) = \delta_{x*}(F)$
  for some fixed $x^* \in \BBR^M$. Else set
\begin{equation*}
{\cal Q}(y,F) = \frac{1}{m(y)} \int_{\Phi^{-1}(y) \cap F} \frac{p(x)}{J_N\Phi(x)} \BH^{M-N} (dx).
\end{equation*}
\end{enumerate}
\label{prop2}

Set $P(x,F) = {\cal Q} (\Phi(x),F)$. Then $P$ is a regular conditional
probability for $P(dx) = p(x) \lambda^M (dx)$ given ${\cal
  C}=\Phi^{-1}({\cal B})$
with ${\cal B}$ the Lebesgue measurable sets in $\BBR^N$.
\end{prop}
\begin{proof}
Clearly (\ref{list.a}) and (\ref{list.b}) are satisfied. To show (\ref{list.c}), fix $C \in {\cal
  C}$ and $F$ a Lebesgue measurable set in $\BBR^M$. Take $g$ in (2.1)
to be
\begin{equation*}
\delta_{C \cap F} (x) \frac{p(x)}{J_N\Phi(x)} \quad \text{ with
  $g(x)$ defined as $0$ if $p(x)=0.$}
\end{equation*}
Where $\delta_{C \cap F}$ denotes the indicator function of
the intersection $C \cap F$.

\noi The co-area formula shows
\begin{align*}
P(C \cap F) = \int_{C \cap F} p(x) \lambda^M (dx) &= \int_{\BBR^N}
\int_{\Phi^{-1}(y)} \delta_C (x) \delta_F (x) \frac{p(x)}{J_N\Phi(x)} \BH^{M-N}
(dx) \lambda^N (dy)\\
&= \int_C \int_{\Phi^{-1}(y)\cap F} \frac{p(x)}{J_N\Phi(x)} \BH^{M-N}
(dx) \lambda^N (dy).
\end{align*}

\noi Let $C_0 = \{ y: m(y) = 0 \}, C_\infty = \{ y: m(y) = \infty
\}, C^+ = (C_0 \cup C_\infty)^C$. Taking $C = F = \BBR^M$, we see
$\lambda^N (C_\infty) = 0$. For $y \in C_0, \int_{\Phi^{-1}(y) \cap F}
\frac{p(x)}{J_N\Phi(x)} \BH^{M-N} (dx) = 0$. Hence, the integrals equal
\begin{align*}
&\int_{C \cap C^+} \int_{\Phi^{-1}(y)\cap F} \frac{p(x)}{J_N\Phi(x)} \BH^{M-N}
(dx) \lambda^N (dy)\\
= &\int_{C \cap C^+} \frac{m(y)}{m(y)} \int_{\Phi^{-1}(y)\cap F} \frac{p(x)}{J_N\Phi(x)} \BH^{M-N}
(dx) \lambda^N (dy)\\
= &\int_C m(y) {\cal Q} (y,F) \lambda^N (dy)\\
= &\int_C P(x,F) P(dx)
\end{align*}
\end{proof}

\noi \textit{Remark} \quad Of course, $m(y)$ can be $0$,  if $\Phi^{-1}(y)$ is
empty or $p$ vanishes there. Similarly, $m(y)$ can be infinite:
consider (following Tjur [1972, Sect. 30]) a set of finite area in
$\BBR^2$ of the shape 
shown in Figure \ref{fig:shape}.
\begin{figure}[h!]
\begin{center}
\includegraphics[width=4in]{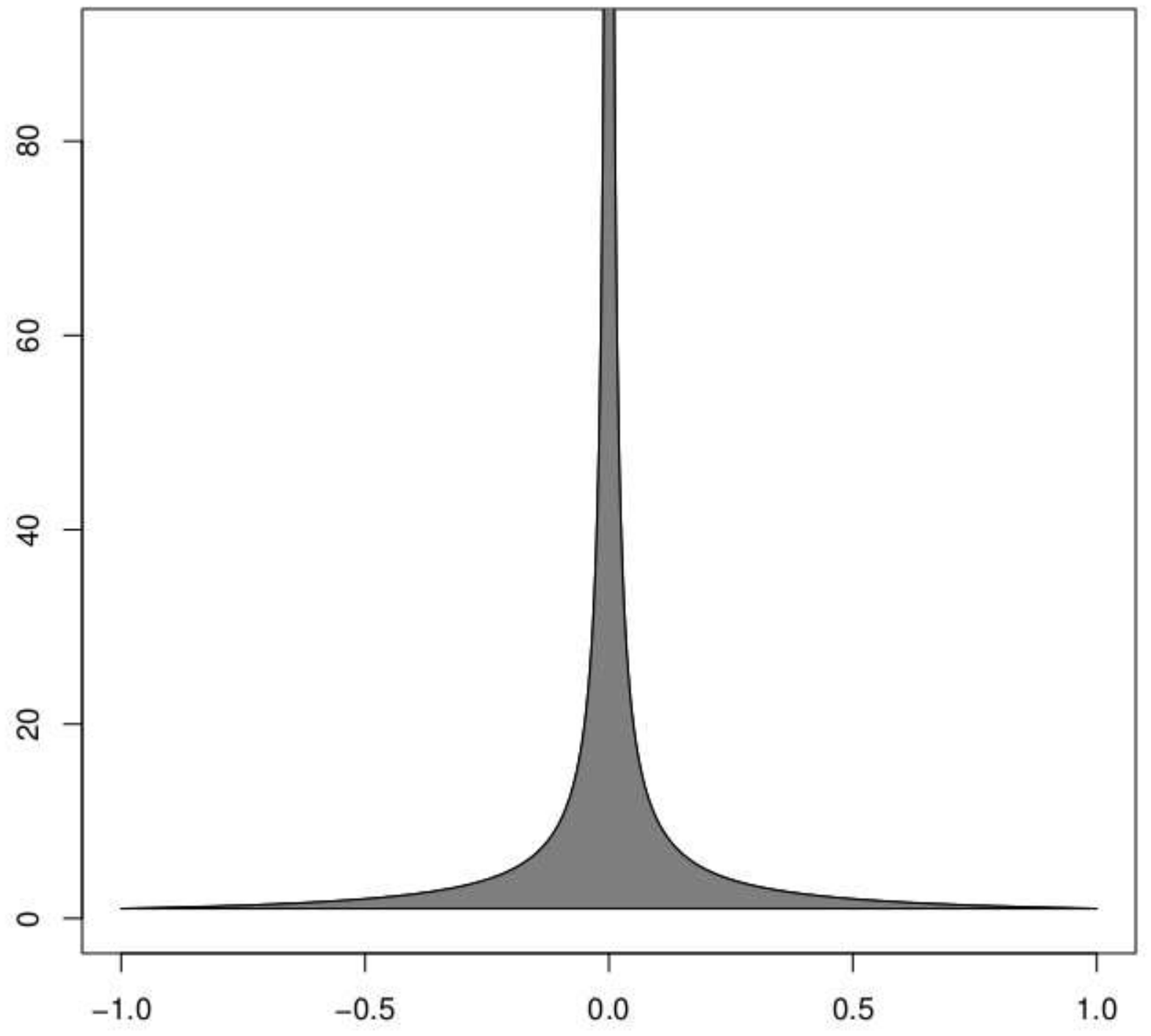}\\
\caption{ Instance of Infinite region.}
 \label{fig:shape}
 \end{center}
\end{figure}
\noi Let $p(x)$ be the normalized indicator of this set. Let
$\Phi(x,y) = x$, so $J_N\Phi(x)=1$. Then $m(0) = \infty.$

\noi{\bf Example 1A (continued)}: From (\ref{eq:1}) the torus is $\{(x,y,z) \in R^3 \}$
$$
x=  ( R+r \cos ( \theta ) ) \cos ( \psi ) , \;\;\; y=( R+r \cos ( \theta ) ) \sin ( \psi ), \;\;\;  z= r \sin ( \theta ) $$
$0 \leq \theta,\psi <2\pi$ for $R > r> 0$. 
What is
the conditional distribution in $(\theta,\psi)$ space given that $x=0$?
In the notation of Proposition \ref{prop2},
$$
p(\theta,\psi)=\begin{cases}
\frac{1}{2\pi}(1+\frac{r}{R}\cos(\theta))  & 0 \leq \theta,\psi <2\pi\\
0 & \mbox{ elsewhere}
\end{cases}
$$
The function $\Phi:R^2\rightarrow R$ is
$$\Phi(\theta,\psi)=
(R+r\cos(\theta))\cos(\psi)
$$
Thus
$$(J\Phi)^2=(r\sin(\theta)\cos(\psi))^2+((R+r\cos(\theta))\sin(\psi))^2
$$
$$
\Phi^{-1}(0)=\{ (\theta,\psi),\;\; 0\leq \theta < 2\pi , \;\; \psi \in \{\frac{\pi}{2},\frac{3\pi}{2}\}\}
$$
It follows that $J\Phi(\theta,\frac{\pi}{2})=J\Phi(\theta,\frac{3\pi}{2})=
R+r\cos(\theta)$. This is proportional to $p(\theta,\psi)$ and Proposition 2b
says that the conditional distribution is uniform on the two line
segments that make up
$\Phi^{-1}(0)$ and assigns equal mass to each segment.

\paragraph{Example 1B (continued)}
Consider the area measure on $\CM^+$
of (\ref{eq2.5}). Proposition \ref{prop2.1}
above
shows that $\CM^+$ is parametrized by a map $f$ from the set $U \subset R^{n-2}$
and gives an explicit expression for the corresponding probability density.
One standard method for sampling from this density is to use the Gibbs sampler.
This entails sampling from the conditional distribution given
the values at some of the coordinates.
One simple implementation which uses Proposition \ref{prop2}
is this:
$\CM^+$ is given as an embedded manifold in $\mathbb{R}^n$.
From $(x_1,x_2,\ldots,x_n)\in \CM^+$, choose three coordinates uniformly at random,
fix the remaining $(n-3)$ coordinates. The map $f$ of Proposition 
\ref{prop2.1}
composed with the projection 
onto the corresponding $(n-3)$ space gives a map $\Phi : {\cal U }\longrightarrow \BBR^{n-3}$.
The conditional density given $\Phi=y$ is explicitly given by Proposition 
\ref{prop2}. Here $\Phi^{-1}(0)$ is a one dimensional curve and the sampling problem reduces to a standard task. We omit further details.


\medskip

\medskip

\noi\textbf{Example 3C: How Not To Sample} \quad Here is a mistake to
avoid. Let $\CM$ be a compact embedded manifold. To sample from the
area measure, the following scheme presents itself. Suppose that for
each point $x \in \CM$ a neighborhood $\CN_x \subseteq \CM$ is specified
(e.g., a ball of specified radius on $\CM$). Suppose it is possible to
sample from the area measure restricted to $\CN_x$. It seems plausible
that this drives a Markov chain with area measure globally. This is an
error. Perhaps the easiest way to see through the problem is to
consider the discrete case:

Consider a finite connected undirected graph with vertex set ${\mathcal X}$ and edge set
$\cal E$. Let $\pi(x) >0,\ \sum_{x \in {\mathcal X}} \pi(x) = 1$ be a
probability distribution $\cal X$. Suppose for each point $x \in
{\mathcal X}$, a neighborhood $\CN_x$ is defined. These may be arbitrary
finite sets; we do not need $x \in \CN_x$, but will assume $y \in \CN_x
\leftrightarrow x \in \CN_y$. For example, we may take $\CN_x =
B_r(x)$, the $r$-ball using graph distance. A Markov chain on $\mathcal X$
is defined as follows: 

From $x$, choose $y \in \CN_x$ with probability $\pi$ restricted to
$\CN_x$. Thus
\begin{equation}
K(x,y)=\begin{cases}
\frac{\pi(y)}{\pi(N_x)} & \text{ if } y \in \CN_x\\
0 & otherwise \end{cases}
\label{exam3.1}
\end{equation}

\begin{lemma}
The chain (\ref{exam3.1}) is reversible with reversing measure
\begin{equation}
\sigma(x)= \frac{\pi(\CN_x) \pi(x)}{z}, \text{ with $z$ a normalizing
  constant.}
\label{exam3.2}
\end{equation}
\end{lemma}
\smallskip

\begin{proof}
If $K(x,y) = 0$, then $K(y,x) = 0$, so reversibility holds. Otherwise
\begin{equation*}
\sigma(x) K(x,y)= \frac{\pi(\CN_x) \pi(x)}{z}
\frac{\pi(y)}{\pi(\CN_x)} = \frac{\pi(x) \pi(y)}{z} = \sigma(y) K(y,x).
\end{equation*}
\end{proof}

\noi\textit{Remarks}

1. Thus, unless $\pi(\CN_x)= $ constant, $\sigma(x) \neq \pi(x)$.

2. In the continuous setting, sampling locally from area measure $\BH$, this
chain has stationary density proportional to $\BH(\CN_x)$. An analysis
of rates of convergence for this walk on compact Riemannian manifolds
in \cite{Lebeau}.

3. On a curve, with distance measured by arc length, $\BH (B_r(x))$
{\em is} constant for $r$ suitably small because of the volume of
tubes theorem. However, this is no longer true for higher-dimensional
manifolds with non-constant Gaussian curvature.

4. We may use the Metropolis algorithm to change the stationary
distribution from $\sigma$ in (\ref{exam3.2}) to $\pi$. The chain is
$\CM(x,y) = \pi(y) \min ( \frac{1}{\pi(\CN_x)}, \frac{1}{\pi(\CN_y)} )$ for $x \neq y \in
\CN_x$. Note that this requires knowledge of $\pi(\CN_x), \ \pi(\CN_y)$.

\section{Exponential Families, Conditional Densities and the Co-Area
  Formula}\label{EXP}

One motivation for the current work is conditional testing in
statistical problems. This is a central topic of classical statistics
beginning with R.~A. Fisher's exact test for independence  in
contingency tables and the Neyman--Pearson theory of uniformly most
powerful unbiased tests for exponential families. The best general
reference for these topics is 
\cite[Chap. 4, 5, 10]{lehmann2005testing}
See also the
survey in \cite{DiaconisSturmfels} and the
techniques and references in \cite{Lindqvist,Lindqvist2006}. 

The problems
addressed in the present paper are a continuous analog. Section
\ref{EXP.a} below presents exponential families in a version
convenient for applications. Section \ref{EXP.b} briefly discusses
conditional densities and sufficiency. Section \ref{EXP.c} uses the
co-area formula to give a useful expression for the conditional
density, given a sufficient statistic, with respect to the area
measure. These formulae are applied in Section 4.

\subsection{Exponential Families}\label{EXP.a}

Many widely-used families of probability measures, such as the Gamma
family of Example 1B, have a common exponential form. Theorems and
properties can be derived generally and then applied in specific
cases. A good first reference for this material is \cite[Sect.~2.7]{lehmann2005testing}.
The specialist monographs of \cite{barndorff1979information}, \cite{Brown:1986} and \cite{Letac} 
may be supplemented by the references in \cite{DKSC} to give an
overview of this basic subject.

Let $T: \BBR^a \to \BBR^b$ be a measurable function. Let
$\Theta  \subseteq \BBR^b$ be a non-empty open set and
$\psi: \Theta  \to \BBR^b$ a measurable function. Let $f(x): \BBR^a \to
\BBR_+$ be measurable and suppose
\begin{equation*}
0 < z(\theta)= \int_{\BBR^a} f(x) e^{\psi(\theta)\bullet T(x)} \lambda^a
(dx) < \infty \text{ for each }\theta \in \Theta .
\end{equation*}

\noi\textbf{Definition} \quad The family of probability densities
\begin{equation}
P_\theta(x)=z^{-1}(\theta)f(x)e^{\psi(\theta) \bullet  T(x)} \qquad \theta \in \Theta 
\label{eq4.1}
\end{equation}

\noi is called the exponential family generated by $(f, \Theta , \psi, T)$.

For the Gamma family in Example 1B, $a=1, b=2, T(x)= \begin{cases}
  (x, \log x) & x>0\\ 0 & otherwise \end{cases}$
\begin{equation*}
\Theta =\BBR^2_+ = \{ (\sigma, a): \sigma, a > 0 \}, \quad
\psi(\sigma, a) = (-\frac{1}{\sigma}, a-1), \quad z(\theta)= \sigma^a\Gamma(a) \quad f(x)= \begin{cases} 1
  & \text{if } x>0\\ 0 & otherwise\end{cases}.
\end{equation*}

The exponential families here are a subclass, in having absolutely
continuous densities whose support does not depend on $\theta$.

\subsection{Sufficiency}\label{EXP.b}

The product measure on $(\BBR^a)^n$ generated by $P_\theta$ of (\ref{eq4.1}) has
density
\begin{equation*}
z(\theta)^{-n} \prod_{i=1}^n f(x_i) {e^{\psi(\theta)\bullet \sum_{i=1}^n T(x_i)}}.
\end{equation*}

\noi The function $\bar{T}= \sum_{i=1}^n T(x_i)$ is called a
sufficient statistic for the family. The references above  show that the distribution of the product measure
conditional on $\bar{T}$ does not depend on $\theta$. Conversely, the
Koopman--Pitman--Darmois theorem says if $P_\theta$ is a family of
measures on $\BBR^a$ with $T$ locally Lipschitz and for some $n \geq
2$, the distribution of $P_\theta^n$, conditional on $T$, does not
depend on $\theta$, then $P_\theta$ is an exponential family. See \cite{Hipp}
 for a careful statement; see \cite{diaconis1988sufficiency} for background and further
references on sufficiency.

For the Gamma family, $\bar{T}$ is equivalent to $S=\sum_{i=1}^n x_i,
P=\prod_{i=1}^n x_i$ as used thoughout.

\subsection{Conditional Densities and the Co-Area Formula}\label{EXP.c}

This dual to the area formula is explained in Section \ref{DEF.c} above. We may
use it directly to compute an expression for the conditional density
of an exponential family given a sufficient statistic.

\begin{thm}
  With notation as above, for $na > b$ consider an exponential family
  (\ref{eq4.1}) based on a Lipschitz $T: \BBR^a \to \BBR^b$. Let
  $\bar{T}: \BBR^{na} \to \BBR^b (= \sum_{i=1}^n T(x_i))$ and suppose
  $J_b\bar{T}(x)\neq 0$ for $\prod f(x_i) \neq 0$. Define $\CM_t=\{
  \bx \in (\BBR^a)^n :\; \bar{T}(\bx) =t \}$ Then, the conditional density on $\CM_t$
  with respect to the area measure is
\begin{equation}
W^{-1} \prod_{i=1}^n f(\bx_i)/J_b\bar{T}(\bx).
\label{eq4.2}
\end{equation}

\noi with the normalizing constant $W=W_t$ taken to be $\int \prod_{i=1}^n
f(\bx_i)/J_b\bar{T}(\bx) \BH^{M-N}(d\bx)$ if this integral is in $(0,\infty)$.
\end{thm}

\begin{proof}
In the co-area formula take $\Psi = \bar{T}: (\BBR^a)^n \to
\BBR^b$. Thus $M=na, N=b$. For $h : \BBR^M \to \BBR^N$ bounded
continuous, set
\begin{equation*}
g(\bx)=
\begin{cases}
\displaystyle\frac{h(\bx)}{J_N\Psi(\bx)} \displaystyle\prod_{i=1}^n P_\theta(x_i) & \text{ if }
J_N\Psi(\bx) \neq 0\\
0 & \text{ otherwise}
\end{cases}\ .
\end{equation*}

\noi Then $\Psi^{-1}(t)=\CM_t$ and the co-area formula shows that
$\CM_t$ has positive, finite total area measure for $\lambda^N$
a.e.t. Further
\begin{equation*}
\int h(\bx) \prod_{i=1}^n P_\theta(\bx_i) \lambda^M (d\bx) =
\int_{\BBR^b} e^{\Psi(\theta) \cdot t} z(\theta)^{-n} \int_{\CM_t}
\frac{h(\bx) \prod_{i=1}^n f(\bx_i)}{J_N\Psi(\bx)} \BH^{M-N}(dx) \lambda^N(dt).
\end{equation*}

\noi This formula says that (\ref{eq4.2}) is a regular conditional
probability for the product measure $\prod_{i=1}^n P_\theta(\bx)$
given $\bar{T}=t$.
\end{proof}

\noi\textit{Remarks}

1. Since the conditional density (\ref{eq4.2}) does not depend on
$\theta$, $\bar{T}$ is a sufficient statistic.

2. The calculation shows that the marginal density of $\bar{T}$ is $e^{\Psi(\theta) \cdot t} / z(\theta)^n W$ with
respect to $\lambda^b(dt)$. Thus the induced measures of $\bar{T}$
form an exponential family.
\medskip

\noi\textbf{Example: Gamma Family} \quad With $\bar{T}: \BBR^n \to
\BBR^2$ given by $\bar{T}(\bx) = (\sum_{i=1}^n x_i, \sum_{i=1}^n \log
x_i)$, for $n > 2$, 
\begin{equation*}
D\bar{T}(\bx)= 
\begin{bmatrix}
1 & 1 & \hdots & 1\\
\frac{1}{x_1} & \frac{1}{x_2} & \hdots & \frac{1}{x_n}
\end{bmatrix}, \quad J_2^2\bar{T}(\bx) = \sum_{i<j} \Big ( \frac{1}{x_i} -
\frac{1}{x_j} \Big) ^2.
\end{equation*}

\noi From Theorem 1, we may sample from the conditional
distribution of the Gamma family given $\bar{T}=t$ on $\CM_t$ by
sampling from the probability density (w.r.t. $\lambda^{n-2}$)
proportional to
\begin{equation*}
\frac{J_{n-2}f(x_3, \ldots, x_n)}{J_2\bar{T}(f(x_3, \ldots, x_n))}
\end{equation*}

\noi on $\CU$ and  $f$ defined in (\ref{2.6}), followed by randomizing the
first two coordinates by a fair coin toss.
\medskip

\section{Neyman's Smooth test and the Gibb's sampler}
This section illustrates a useful general procedure (the Gibbs
sampler) in a natural example: Neyman's smooth test for goodness of
fit. The problem reduces to sampling from an explicit density
$f(x_1,x_2,\ldots,x_n)$ on the following submanifold: fix $m$ and $p_1 \geq p_2 \geq \cdots p_m$. Let
$$\Mp=\{x_1,x_2,\ldots x_n, 0\leq x_i\leq 1, \sum_{j=1}^n x^i_j=p_i,1 \leq i\leq m
 \}.$$
 In Neyman's case, $m=4$, assume this for now. The idea  underlying our algorithm, 
 developed below, is
to pick a uniformly chosen subset
of $m+1=5$ coordinates with probability $1/{n \choose 5}$,
say the first five.
Set
$\bar{p}_i=\sum_{j=1}^5x_j^i$. The submanifold
 \begin{equation}
  {\mathcal M}_{\bar{p}}
 =\{ 
x_1,x_2,\ldots x_5, 0\leq x_i\leq 1, \sum_{j=1}^5x^i_j=\bar{p}_i,1 \leq i\leq 4
\}
 \label{eq5.0}
\end{equation}
is a curve 
which lies both on the submanifold  ${\Mp}$ and in $\mathbb{R}^5$.
We may sample from the conditional density
on the curve and replacing
$(x_1,x_2,\ldots,x_5)$ by the sampled values
gives a new point on
$\Mp$.

 Repeatedly choosing fresh five-tuples gives a connected
reversible Markov chain on $\Mp$ with $f$ as its stationary density .
In the present section we find it convenient
to work directly with the density of $f$ with respect to the area
measure, avoiding the extra step
of local coordinates.
Neyman's smooth test is developed in 4.1, the relevant conditional densities
are derived in 4.2 and 4.3 contains a study
of the ergodicity of
 this chain.
Section 4.4 develops an idea of \cite{Besag}  for valid testing 
with non ergodic chains.
\subsection{Neyman's Smooth test}
\label{sec4.1}
Consider investigating the following null hypothesis; 
fix $F$ a distribution  function of a continuous random variable. Let
\begin{equation}
H_0: X_1,X_2,X_3,\ldots X_n \sim iid \quad F
\end{equation}
Then
$$Y_i=F(X_i)$$ are iid uniform on $[0,1]$. \cite{Neyman} developed a test of $H_0$ based on
testing $\theta=0$ in the model
\begin{equation}
\label{eq5.2}
f_\theta(y)=z^{-1}e^{\theta_1y+\theta_2y^2+\theta_3y^3+\theta_4y^4},
\qquad 0\leq y\leq 1
\end{equation}
This test (and its modifications by \cite{FNDavid,Barton,Barton:1956})
 has been shown to have a good empirical track record and comes
in for repeated favorable
mention in Lehman and Romano's  survey of testing goodness of
fit \cite[chapter 9]{lehmann2005testing}.
That chapter also explains the difficulty of such omnibus testing problems.
One justication for this test is that if the data are from a {\em smooth} distribution $F$,
using a simple $\chi^2$ test loses information because it breaks the data into categorical bins,
losing the actual ordering of the bins.

Any smooth positive probability density $h(y)$ on $[0,1]$ can be expanded as
$$
h(y)=e^{log h(y)}=e^{\sum_{i=0}^\infty \theta_iy^i}
$$

The four parameter exponential family is a commonsense truncation of this non-parametric
model.  \cite{Fan} has developed tests based on $m$ term approximations with $m$ chosen adaptively from the data.

In the rest of this section we investigate the adequacy of the truncation (\ref{eq5.2}) (with $m=4$)
by testing if the model (4.3) fits the data.
 Thus given data $Y_1,Y_2,\ldots Y_n$ in $[0,1]$, we develop
conditional tests of the model (\ref{eq5.2}). These ideas work for every $m$ and could be used as input to Fan's adaptive procedure.
The four dimensional sufficient statistics for the family
(\ref{eq5.2}) is
$$
\mathbf{p}=(p_1,p_2,p_3,p_4),\qquad p_i=\sum_{j=1}^nY^i_j
$$
The conditional procedures explained in section 4.2
are based on the conditional distribution of the model $f_\theta$ given $\mathbf{p}$. This is supported on
\begin{equation}
\Mp=\{ (x_1,x_2,\ldots,x_n), 0\leq x_i \leq 1, \sum_{j=1}^n x_j^i =p_i, 1\leq i \leq 4 \}
\label{eq4.3}
\end{equation}
This is a compact $n-4$ dimensional submanifold of $[0,1]^n$.
To actually construct a test, a test statistic must be
chosen. Neyman's test of section \ref{sec4.1} was based on the L$^2$ norm of the averages of the first four orthogonal polynomials for the uniform distribution on $[0,1]$. Under
(\ref{eq5.2})  the sum of these norms should have been an approximate chi-square
(4)  distribution. 
 We may follow Neyman, using a further orthogonal polynomial as the test statistic but calibrating it
 with the exact conditional distribution.

\subsection{The Gibbs Sampler}
The Gibbs sampler is well developed in \cite{LiuBook}.
As usually explained, to sample from a probability density
$g(z_1,z_2,\ldots,z_n)$
on $\mathbb{R}^n$ one begins at a sample point $z_0=(z_1^0,z_2^0,\ldots,z_n^0)$
and changes coordinates sequentially: first  to $(z_1^1,z_2^0,\ldots,z_n^0)$
then to $(z_1^1,z_2^1,\ldots,z_n^0)$ ...then
$z_1=(z_1^1,z_2^1,\ldots,z_n^1)$.
The $ith$ change is made by sampling from the conditional distribution
of the $ith$ coordinate given all the rest. The one dimensional problem
is supposed to be easy to do. The transition from $z^0$ to $z^1$ is one step of the Gibbs sampler.
Proceeding as above to
$z^2$, $z^3$,$\ldots$ gives a Markov chain with $g$
as stationary distribution. 
In the present problem
\begin{enumerate}
\item[(a)]
It is not possible to change just one coordinate and stay on the surface
(\ref{eq5.0}). The minimal change is in five coordinates
resulting in the curve 
\begin{equation}
\{ (x_1,x_2,\ldots,x_5): 0\leq x_i  \leq 1 \;\; \sum_i^5=p_j
\}.
\label{eq4.4}
\end{equation}
\item[(b)]
Instead of random sampling, one can systematically run through all sets of five coordinates
using for instance a Gray code approach as in \cite{GrayCodes}. 
\item[(c)]
Sampling from a conditional distribution on the curve in $(a)$ is not so simple and instead a single Metropolis step is proposed. This is sometimes 
called `Metropolis on Gibbs' in the literature, for notational clarity we suppose that 
the five chosen coordinates are the first five. Let $P$ be the conditional distribution
for the model
(\ref{eq5.2}) on the submanifold (\ref{eq5.0}). Let $Q$ be the conditional measure on the curve
(\ref{eq4.4}). The following proposition determines the density of $Q$ with respect to arc-length.
\end{enumerate}
\begin{prop}
\label{prop5.1}
The measure $Q$ on the curve (\ref{eq4.4})  has density
with respect to arc-length 
$$q(x_1,x_2,x_3,x_4,x_5)=z^{-1}\sqrt{J_4^{-1}}\qquad z^{-1} \mbox{ a normalizing constant}$$
\begin{equation}
\label{eq5.4}
J_4=
\mbox{det}
\begin{pmatrix}
5& 2\bar{p}_1&3\bar{p}_2 & 4\bar{p}_3  \\
2\bar{p}_1 & 4\bar{p}_2 & 6\bar{p}_3& 8\bar{p}_4 \\
3\bar{p}_2 & 6\bar{p}_3 & 9\bar{p}_4 & 12\bar{p}_5\\
4\bar{p}_3 & 8\bar{p}_4 &12\bar{p}_5& 16\bar{p}_5\\
\end{pmatrix}
\qquad \bar{p}_i=\sum_{j=1}^5 x_j^i, 1\leq i \leq 5
\end{equation}
\end{prop}
\proof{
By the usual calculus of double conditioning, $Q$ is the conditional distribution
of the product measure $f_\theta^5$ on $[0,1]^5$
 given $\bar{p}_1,\bar{p}_2,\bar{p}_3,\bar{p}_4$. Now use Theorem 1 of section 3.3.
  The mapping $\bar{T}:[0,1]^5 \rightarrow
\mathbb{R}^4$ takes
$T(y_1,y_2,y_3,y_4,y_5)=(\bar{p}_1,\bar{p}_2,\bar{p}_3,\bar{p}_4)$.
Clearly the $5 \times 4$
derivative $D\bar{T}$ is
$$
D\bar{T}=
\begin{pmatrix}
1 & 2y_1 & 3y_1^2 &4y_1^3\\
1 & 2y_2 & 3y_2^2 &4y_2^3\\
1 & 2y_3 & 3y_3^2 &4y_3^3\\
1 & 2y_4 & 3y_4^2 &4y_4^3\\
1 & 2y_5 & 3y_5^2 &4y_5^3\\
\end{pmatrix}
$$
so that $J_4$ is given by 
(\ref{eq5.4}) as claimed.
}

\noi{\bf Remark}:\\
For general $m$, the density is proportional to $J_m^{-\frac{1}{2}}$
with $J_m$ having $i,j$ entry $i\cdot j p_{i+j-2}, 1\leq i,j\leq m$
The following algorithm combines the ideas above
to
give a reversible Markov chain for sampling from the conditional distribution of
the model \ref{eq5.2} on the manifold $\Mp$. From $\mathbf{x} \in \Mp$
 
\begin{enumerate}
\item Choose five coordinates uniformly at random. Without loss, suppose these are the first
five, calculate $\bar{p}_i=\sum_{j=1}^5 x_j^i,\qquad 1\leq i \leq 4$.
\item Pick a small parameter $\epsilon$, then
choose one of the five coordinates uniformly at random without loss, suppose the first coordinate
has been chosen. Change $x_1$ to $y_1=x_1+\epsilon_1$
with $\epsilon_1$ chosen uniformly in $[x_1-\epsilon,x_1+\epsilon]$. Solve  for $y_2,y_3,y_4,y_5$
so that $\mathbf{y}=(y_1,y_2,y_3,y_4,y_5) \in   {\mathcal M}_{\bar{p}}
$ as in (\ref{eq5.0}).
\item
Calculate $J_4(\mathbf{x})$, $J_4(\mathbf{y})$ from Proposition 3 above.
If $J_4(\mathbf{x})\geq J_4(\mathbf{y})$ the algorithm moves to $\mathbf{y}$.
If $J_4(\mathbf{x})< J_4(\mathbf{y})$
flip a coin with success probability
$$\sqrt{\frac{J_4(\mathbf{x})}{J_4(\mathbf{y})}}
$$
If success move to $\mathbf{y}$, otherwise stay at $\mathbf{x}$
\end{enumerate}
Remarks:
\renewcommand{\labelenumi}{\arabic{enumi}}
\begin{enumerate}
\item  For $m \leq 4$, calculations for solving the $\mathbf{y}$ can be done in closed
form as they involve at most quartic equations. For higher $m$ a variety
of numerical procedures are available.
\item Of course, if $\mathbf{y}$ in step (b) is outside $[0,1]^5$, the algorithm stays
at $\mathbf{x}$ .
\item We began studying the problem hoping to parametrize the curve
(\ref{eq5.0}) and sample directly from the arc length measure. This proved
impractical. The technique we have developed
seems easier and is applicable to general continuous
exponential families.
%
\end{enumerate}

\subsection{Ergodicity}
Let $P_j(x)=x_1^j+\ldots +x_n^j$ and $S$ be the set defined by
\begin{equation}
\label{eq:sample1}
0<x_1<\ldots <x_n<1,~~P_1(x)=c_1, \ldots, P_4(x)=c_4.
\end{equation}
The closure of $S$ will be denoted by $\bar{S}$.
We also assume that $1\ge c_1>c_2> c_3 >c_4>0$ which is a necessary condition for the existence of a solution to (\ref{eq:sample1}).  Assume that the system (\ref{eq:sample1}) has a solution.

\begin{lemma}
\label{lem:sample1}
Let $y\in S$ be a solution to (\ref{eq:sample1}).  Then there is a submanifold of dimension $n-4$ passing through $y\in S$.  Furthermore the orthogonal projection of $S$ near $y$ on any coordinate line $x_j$ contains a neighborhood of $y_j$.
\end{lemma}

\begin{proof}
 We have $dP_j(x)=jx_1^{j-1}dx_1+\ldots +jx_n^{j-1}dx_n$.  Therefore to show the first assertion it suffices to show that the matrix
\begin{eqnarray*}
\begin{pmatrix}
1&1&\ldots &1\\ x_1&x_2&\ldots &x_n\\ x_1^2&x_2^2&\ldots &x_n^2\\ x_1^3&x_2^3&\ldots &x_n^3
\end{pmatrix}
\end{eqnarray*}
has rank 4 which is immediate.  The second assertion follows from the fact that the locally the system can be solved near $y$ as function of any $n-4$ coordinates. 
\end{proof}

\begin{lemma}
\label{lem:sample2}
Let $k\ge 6$ and $y\in S$ be a solution to (\ref{eq:sample1}).  Consider the solution of the system (\ref{eq:sample1}) subject to the additional requirements
\begin{eqnarray*}
x_j=y_j,~~~{\rm for}~~j\ge k.
\end{eqnarray*}
Then there is a submanifold of dimension $k-5$ of solutions passing through $y$.  For $k=6$ the solution is a curve and its projection on the coordinate line $x_i$, $1\le i\le 5$ contains a neighborhood of $y_i$.
\end{lemma}

\begin{proof}
We look at the differentials $dP_j$, $j=1,2,3,4$ and $dx_j$, $j\ge k$.  To prove the first assertion it suffices to show that the $(n-k+5)\times n$ matrix
\begin{eqnarray*}
\begin{pmatrix}
1&1&\ldots &\ldots&1\\
x_1&x_2&\ldots &\ldots &\ldots&x_n\\
x_1^2&x_2^2&\ldots &\ldots&\ldots &x_n^2\\
x_1^3&x^3_2&\ldots &\ldots &\ldots&x_n^3\\
0&0&\ldots&1~0&\ldots&0\\
0&0&\ldots&0~1&\ldots&0\\
\vdots&\vdots&\ddots&\vdots~\vdots&\ddots&\vdots\\
0&0&\ldots&0~0&\ldots&1\\
\end{pmatrix}
\end{eqnarray*}
has rank $n-k+5$ which is obvious.  The second assertion follows from the fact we can solve for $n-1$ coordinates in terms of any one of $x_i$'s for $i=1,2,3,4$.  
\end{proof}

Let $M\subset S$ be a connected component of $S$.  We consider the following process in $M$.  Given that the process is at $y=(y_1,\ldots,y_n)\in M$, one chooses five coordinates $i_1,\ldots,i_5$ and the process can move to any point along the curve defined by
\begin{eqnarray*}
x_j=y_j,~~~{\rm for}~~j\ne i_1,\ldots,i_5.
\end{eqnarray*}
The question we want to answer is whether any two points in $M$ {\it communicate} in the sense that one can move from one to the other in a finite number of iterations.
More technically, we say two points, $y,z$ are sufficiently close if, given $y\in M$ there is $\delta>0$ such that if $z$ is within $\delta$ of $y$, then one can move from $y$ to $z$ 
in finite number of iterations.  The positive number $\delta$ may depend on $y$.
\begin{lemma}
\label{lem:sample3}
If two points $y,z\in M$ are sufficiently close then they communicate. 
\end{lemma}

\begin{proof}
We do induction on $n$.  The case $n=5$ is clear.  Let $i_1,\ldots,i_5=1,\ldots,5$ and $k=6$ in the notation of Lemma \ref{lem:sample2}.  Then the determinant of the matrix in the proof of Lemma \ref{lem:sample2} is
\begin{eqnarray*}
\pm \prod_{i<j<6}(x_i-x_j).
\end{eqnarray*}
Therefore if $y$ and $z$ sufficiently close then one can move from $(y_1,\ldots,y_n)$ to a point \\ $(z_1,y_2^\prime,\ldots,y_5^\prime,y_6,\ldots,y_n)$ by the second assertion of Lemma \ref{lem:sample2}.  Now the induction hypothesis applies to complete the proof. 
\end{proof}

\begin{prop}
\label{prop:sample1}
Any pair of points in $M$ communicate.
\end{prop}

\begin{proof}
 Starting at $y\in M$ we show that the set of points in $M$ that can be reached in a finite number of steps from $y$ is both open and closed in $M$.  The required result then follows from connectedness of $M$. From Lemma \ref{lem:sample3} it follows that the set of points that can be reached from $y$ in finitely many iterations is open.  To show closed-ness let $y=y^{(1)},y^{(2)},\ldots $ be a sequence of points each of which can be reached in finitely many steps from $y$ and assume $y^{(m)}\longrightarrow z\in M$.  Then for all $m$ sufficiently large the point $y^{(m)}$ lies is a sufficiently small neighborhood of $z$ and Lemma \ref{lem:sample3} is applicable to show that $z$ can be reached in finitely many steps from such $y^{(m)}$ proving the required `closed-ness'.  
 \end{proof}

Let $S^\prime$ be the set defined by

\begin{equation}
\label{eq:sample11}
0\le x_1,\ldots ,x_n\le 1,~~P_1(x)=c_1, \ldots, P_4(x)=c_4,
\end{equation}
and $M^\prime$ be a connected component of $S^\prime$.  We consider the process in $M^\prime$ where in addition we allow any permutation of the coordinates as well as evolution described in $M$.

\begin{prop}
\label{prop:sample2}
Any pair of points in $M^\prime$ communicate.
\end{prop}

\begin{proof}
For points away from the set $V$ consisting of the boundary hyperplanes of the unit cube in $\mathbb{R}^n$ and the generalized diagonal $\bigcup_{i\ne j}\{x_i=x_j\}$ the assertion follows from Proposition \ref{prop:sample1}.  Applying the Curve Selection Lemma (see for example \cite{Milnor}
) we move away from $V$ in one step, and then Proposition \ref{prop:sample1} is applicable. 
\end{proof}

\subsection{Valid  tests and connectedness}
For many applications of the present techniques, it is only a conjecture that the algorithms are ergodic.
Consider the manifold $\Mp$
above based on the first four sample moments. Choosing 5 coordinates
and sampling from the correct conditional distribution on the resulting curve 
gives a way of moving around on $\Mp$. However it has not been proved that this algorithm
is connected; 
Indeed Proposition \ref{prop:sample2}
of section 4.3 only shows that the algorithm
goes between points in the same connected component (in the topological sense)
in finitely many steps.

\cite{Bormeshenko} gave a difficult proof that the analogous problem based on
changing 3 coordinates  on the manifold determined by 
the sum and the sum of squares is connected and we certainly conjecture
this for any number of moments.

If these samples are used for goodness of fit test, 
there is a valid test procedure available, even in the absence of connectedness, by adapting an
idea of \cite{Besag}.

The idea is simple. Let ${\cal X}$ be the original data. This gives rise   to a point $x^*_0$ on $\Mp$.
Suppose $K(x,dy)$ is a Markov chain with the correct stationary distribution on the connected component containing $x^*_0$.
Fix a number of steps $T^*$ and run this chain $T^*$ steps to get $y^*$  say. Then run the time reversed chain, starting at $y^*$ for $T^*$ steps and independently repeat this $B^*$ times (starting at $y^*$ each time). This results in $(x_1^*,x_1^*,\ldots,x_{B^*}^*)\in \Mp$. The $B^*+1$ values
$(x_0^*,x_1^*,x_1^*,\ldots,x_{B^*}^*)$
are exchangeable, so the relative position of any test statistic
$s(x_0^*)$ among $s(x^*_i)$ is uniform
under the null hypothesis. If $s(x_0^*)$ is among the extreme values of these
statistics then a valid rejection is possible.

\section*{Acknowledgements}
We thank Hans Andersen, Olena Bormishenko,
Greg Brumfiel, Brian White, Leonid Pekelis and an anonymous referee for help with this paper. 
\bibliographystyle{imsart-nameyear.bst}
\bibliography{samplingm}

\end{document}